\documentclass[a4paper,UKenglish]{article}
\pdfoutput=1
 \usepackage{hyperref}
\usepackage{a4wide}
\usepackage[protrusion=true,expansion=true]{microtype}
\usepackage[utf8]{inputenc}
\usepackage{amsfonts, amsmath, amsthm}
\usepackage{xspace}							
\usepackage{booktabs}
\usepackage{paralist}
\usepackage{authblk}   
\usepackage{url,hyperref}

\usepackage{tabularx}

\theoremstyle{plain}
\newtheorem{theorem}{Theorem}

\newtheorem{corollary}{Corollary}
\newtheorem{definition}{Definition}

\newtheorem{proposition}{Proposition}
\newtheorem{example}{Example}

\newcommand{\R}{\mathbb{R}}

\newcommand{\IR}{\mathbb{IR}} 
\newcommand\ol{\overline}    
\newcommand\ul{\underline}

\DeclareMathOperator{\df}{:=}

\newcommand{\NP}{{\ensuremath{\mathsf{NP}}}\xspace}
\newcommand{\pP}{{\ensuremath{\mathsf{P}}}\xspace}
\newcommand{\coNP}{{$\mathsf{co\text{-}NP}$}\xspace}
\newcommand{\ZDG}{{zero DG}\xspace}
\newcommand{\DG}{{DG}\xspace}
\newcommand{\ILP}{{ILP}\xspace}
\newcommand{\typeA}{{type (A)}\xspace}
\newcommand{\typeB}{{type (B)}\xspace}
\newcommand{\typeC}{{type (C)}\xspace}
\newcommand{\TypeA}{{Type (A)}\xspace}
\newcommand{\TypeB}{{Type (B)}\xspace}
\newcommand{\TypeC}{{Type (C)}\xspace}

\newcommand{\polynomial}{
 polynomial}
\newcommand{\coh}{
 \coNP-hard}

\newcommand{\npc}{
 \NP-complete}
\newcommand{\Eng}[1]{}

\newcommand{\T}[1]{#1{}^\top\!} 

\newcommand{\Int}[1]{\protect{\textrm{\boldmath $#1$}}} 

\renewcommand\mid{ \text{\ensuremath{\mathsf{\ subj.\ to\ }}} }
\renewcommand\min{ \text{\ensuremath{\mathsf{\ min\,}} }}
\renewcommand\max{ \text{\ensuremath{\mathsf{\ max\,}} }}
\newcommand\maxx{ \text{\ensuremath{\mathsf{\ max\,}}}}
\newcommand\minn{ \text{\ensuremath{\mathsf{\ min\,}}}}

\newcommand{\Intt}[1]{\T{\Int #1}} 
\newcommand{\D}[1]{#1_\Delta} 
\newcommand{\diag}[1]{\mathop\mathrm{diag}(#1)} 

\newcommand{\Dm}[1]{\underline #1} 
\newcommand{\Hm}[1]{\overline #1} 
\newcommand{\Hmt}[1]{\T{\Hm #1}} 
\newcommand{\Dmt}[1]{\T{\Dm #1}} 
\renewcommand\ol{\overline}              
\renewcommand\ul{\underline}             

\newcommand{\olt}[1]{\T{\ol #1}} 
\newcommand{\ult}[1]{\T{\ul #1}} 
\newcommand\bfm[1]{\protect{\textrm{\boldmath $#1$}}}

\renewcommand\b {\bfm{b}}
\newcommand\cc {\bfm{c}}

\newcommand\A {\bfm{A}}

\newcommand\At {\T{\bfm{A}}}

\newcommand\bt {\T{\bfm{b}}}
\newcommand\ct {\T{\bfm{c}}}

\newcommand{\IFunc}[1]{#1(\A,\,\b,\,\cc)} 

\newcommand{\IF}{\IFunc{f}}
\newcommand{\IG}{\IFunc{g}}

\newcommand{\IDmF}{\IFunc{\Dm f}}
\newcommand{\IHmF}{\IFunc{\Hm f}}
\newcommand{\IDmG}{\IFunc{\Dm g}}
\newcommand{\IHmG}{\IFunc{\Hm g}}

\newcommand{\brule}{\rule[-1.5ex]{0pt}{0pt}}

\title{Duality Gap in Interval Linear Programming\thanks{
Jana Novotn\'a and Milan Hlad\'ik were supported by the Czech Science Foundation Grant P403-18-04735S.
The student work was supported by the grant SVV–2017–260452.}~\thanks{
Short version was presented at conference SOR'17; see~\cite{sor}.
}
}

\author[1]{Jana Novotná}
\author[1]{Milan Hladík}
\author[1]{Tomáš Masařík}

\affil[1]{Department of Applied Mathematics, Faculty of Mathematics and Physics, Charles University, Prague, 
Czech Republic}
\affil[ ]{\texttt{\{janca,hladik,masarik\}@kam.mff.cuni.cz}}

\date{}

\providecommand{\keywords}[1]{\textbf{Keywords: } #1}

\begin{document}

\maketitle

\begin{abstract}

This paper deals with the problem of linear programming with inexact data represented by real closed intervals.
Optimization problems with interval data arise in practical computations and they are of theoretical interest for more than forty years.
We extend the concept of duality gap (DG), the~difference between the primal and its dual optimal value, into interval linear programming.
We consider two situations: First, DG is zero for every realization of interval parameters (the so called strongly zero DG) and, second, DG is zero for at least one realization of interval parameters (the so called weakly zero DG).

 We characterize strongly and weakly zero DG and its special case where the matrix of coefficients is real. 
We discuss computational complexity of testing weakly and strongly zero DG for commonly used types of interval linear programs and their variants with the real matrix of coefficients. We distinguish the NP-hard cases and the cases that are efficiently decidable.
Based on DG conditions, we extend previous results about the bounds of the optimal value set given by Rohn.
We provide equivalent statements for the bounds.

\end{abstract}


\section{{Introduction}}
The area of interval linear programming connects linear programming with interval analysis, which is one of approaches to deal with inexact data. Optimization problems with interval data are intensively studied for last forty years, see the survey by Hladík \cite{Hla2012a} or the chapters by Rohn \cite{Roh2006:3,Roh2006:2} for more details.
 
We extend the concept of duality gap (i.e., the difference between the~primal and the dual optimal value) from linear programming into the setting of interval linear programming. To our best knowledge, we are the first who study thoroughly this topic in the context of interval linear programming despite the fact that duality gap is one of the basic concepts of linear programming and a~very useful property to describe behavior of programs with respect to their duals. We investigate two situations. At least one particular realization of~interval values has zero duality gap (weakly zero duality gap), or all realizations have zero duality gap (strongly zero duality gap).
 
In particular the second property, strongly zero duality gap, is very useful since it allows us to move from a primal to its dual program,  e.g.\ it is a sufficient assumption to have an equality between optimal value sets of a primal and its dual interval linear program. Our main interest is in characterizations of weakly and strongly zero duality gap and testing their computational complexity.

We focus on interval linear programs and also on their special case --- interval linear programs with the real constraint matrix (degenerated matrix). Even programs with degenerated matrix are studied \cite{GabMur2010} in recent years and often used in practical computation (network flows, transportation problems \cite{CerAmb2017}, etc.).

In the setting of strongly zero duality gap we obtained an equivalent characterization for the lower and the upper bound on the optimal value set. This extends previous results given by Rohn~\cite{Roh2006:3}.

\subsection{Problem formulation}    
    Given two matrices $\Dm A,\; \Hm A \in \R^{m \times n}$ such that $\Dm A \le \Hm A$, we define \emph{interval matrix} $\A$ as the set
 $$\Int A = [\Dm A, \Hm A] = \{A \in \R ^{m\times n} : \Dm A \le A \le \Hm A\},$$
 where $\le$ means element-by-element comparison of the corresponding matrices.
 The set of all interval $m \times n$ matrices is denoted by $\IR^{m\times n}.$

Let $\A\in\IR^{m\times n}$, $\b \in\IR^{m}$ and $\cc\in\IR^{n}$ be given. We define an \emph{interval linear program} (ILP\footnote{Do not confuse it with the same abbreviation ILP for Integer Linear Programming.}) as a family of linear programs 
\begin{align*}
\min { \T cx\mid x \in \mathcal M(A, b)},
\end{align*} 
where $A \in \Int A,$ $b \in \Int b,\; c \in \Int c$, and $\mathcal M(\A, \b)$ is the feasible set described by linear equations and inequalities. We write it in short as 
 \begin{align}
 \min{\T{\Int c}x\mid x \in \mathcal M(\Int A, \Int b)}. \label{generalILP}
\end{align}
For interval arithmetic and an introduction to interval analysis, see, e.g., books \cite{AleHer1983,MooKea2009}.

\paragraph{Duality gap.} 
Duality gap is the difference of optimal values of an primal and its dual program in mathematical programming. In linear programming, duality gap is zero if and only if at least one of the primal or its dual program is feasible; otherwise it is $\infty$.

\begin{definition}
\emph{Duality gap} of $(\ref{generalILP})$ is said to be \emph{weakly zero} if at least one scenario has zero duality gap, and \emph{strongly zero} if all scenarios have zero duality gap.

\end{definition}

\subsection{Notations, preliminaries and state-of-the-art}
We denote the center matrix of $\A$ by $A_c\df\frac{1}{2}(\ul A+\ol A)$ and the radius matrix of~$\A$ by $ \D A\df \frac{1}{2}(\ol A-\ul A)$. We denote by $\diag v$ the diagonal matrix with entries $v_1,\dotsc,v_n $.

A particular realization of interval values, i.e.\ a linear program with real values, is called \emph{a~scenario}. We denote it by $(A,b,c)$, where $A \in \A$, $b\in\b$ and $c\in\cc$. 
We say that ILP is \emph{weakly}, resp. \emph{strongly}, \emph{feasible} if it is feasible for at least one scenario, resp. all scenarios. A solution which satisfies all scenarios, resp. at least one scenario, is called \emph{strong solution}, resp. \emph{weak solution}.
An~interval matrix is said to be \emph{degenerated} if 
$ \Dm A = \Hm A$. 
 We are using the term \emph{ILP with degenerated matrix} which means an ILP with a real constraint matrix and an interval right hand side and an  interval optimization function.

Since different forms of ILP do not have to be equivalent (unlike in linear programming) we have to distinguish between them; see \cite{GarHla2017a,Hla2017a} for detailed discussion. Dual problems in ILP straightforwardly extend dual problems in linear programming, for some results see \cite{Roh1980,Roh2006:3,GabMur2010b,Serafini05}.
Basic types of (\ref{generalILP}) and their dual programs are shown in Table~\ref{tab:dual}. 

We denote the \emph{optimal value set} of (\ref{generalILP}) as the set of optimal values of all scenarios.
Formally, let
\begin{align*}
f(A,b,c) := \min { \T cx\mid x \in \mathcal M(A, b)}
\end{align*} 
be the optimal solution of a scenario with $A \in \Int A,$ $b \in \Int b$ and $c \in \Int c$. Then the optimal value set is
\begin{align*}
f(\A,\b,\cc) := \{f(A,b,c): A \in \Int A,\ b \in \Int b,\ c \in \Int c\}.
\end{align*} 
Analogously, we use $g(\A,\b,\cc)$ for the dual optimal value set. 
Notice that optimal value set $f(\A,\b,\cc)$ does not have to be an interval, it can be also disconnected. The conditions under which the optimal value function is continuous, and therefore the optimal value set connected, were addressed, e.g., in  \cite{Bee1978,MosHla2016a}.
We define lower and upper bounds of~the optimal value set as 
\begin{align*}
\IDmF  & := \inf { f(A,b,c) \mid A \in \Int A,\ b \in \Int b,\ c \in \Int c},\\
\IHmF  & := \sup { f(A,b,c) \mid A \in \Int A,\ b \in \Int b,\ c \in \Int c}.
\end{align*} 
The lower and upper bound cases are sometimes called the best and the worst case.
There always exists a scenario where the bounds are attained, implicitly in~\cite{Hla2013}. 
Thanks to this fact we can use minimum and maximum, defined on the~extended real axis, instead of infimum and supremum, respectively. 

Computation of these extremal optimal values was studied in \cite{Bee1978,ChinRam2000,Hla2009b}, among others. This problem is \NP-hard in general \cite{GabMur2010,Roh1997,Roh2006:3}, but it becomes easy for some special forms or under some stability criteria \cite{Bee1978,Hla2012a,Hla2014a,Kon2001}.

We use commonly known relationships between optimal value sets of a~primal and its dual ILP, which hold thanks to weak and strong duality in linear programming.

\begin{proposition}[Weak duality for ILP] \label{claim:weak_duality}
It holds that
$$\IDmF \ge \IDmG \, \text{ and }\, \IHmF \ge\IHmG .$$
\end{proposition}
\begin{proposition}[Strong duality for ILP]\label{claim:strong_duality}
If duality gap of (\ref{generalILP}) is strongly zero, then $$\IF=\IG.$$
\end{proposition}

Notice that the converse implication in Proposition \ref{claim:strong_duality} does not hold true in general. A~counterexample is shown in Example~\ref{ex:all}.

\begin{table}[t]
\begin{center}
\begin{tabular}{c@{\hspace{0.4cm}}l@{\hspace{1.0cm}}l} 
\toprule
Type & Primal ILP & Dual ILP \\
\midrule
(A) & $\min \Intt cx \mid \Int Ax=\Int b,\; x\ge 0$ & $\max \Intt by \mid \Intt Ay\le \Int c$\brule\\

(B)& $\min \Intt cx \mid \Int Ax\le\Int b$ & $\max \Intt by \mid \Intt Ay = \Int c,\; y\le 0$\brule\\

(C)& $\min \Intt cx \mid \Int Ax\le \Int b,\; x\ge 0$ & $\max \Intt by \mid \Intt Ay\le \Int c,\; y\le 0$\brule\\
\bottomrule
\end{tabular}
\end{center}
\caption{Dual programs.} \label{tab:dual}
\end{table}


\subsection{Our results}

We introduce the concept of \DG to interval linear programming. We give a~full characterization of weakly zero \DG in general  (Prop.~\ref{prop:cond:weak:full}) and derive specific forms for types (A), (B) and (C) (Cor.~\ref{cor:cond:WZDG_A}, \ref{cor:cond:WZDG_B}, \ref{cor:cond:WZDG_C}).
We obtain a full characterization of strongly zero \DG in \ILP with degenerated matrix in general (Thm.~\ref{thm:cond:deg:SZDG_strong}) and for types (A), (B), (C) (Cor.~\ref{cor:cond:deg:SZDG_A}, \ref{cor:cond:deg:SZDG_B}, \ref{cor:cond:deg:SZDG_C}).
We provide sufficient conditions and necessary conditions of strongly zero \DG in \ILP (Prop.~\ref{prop:cond:SZDG_str-feas}, \ref{prop:cond:connectivity} and \ref{thm:gap:condit:strong3}). 

We determine the computation complexity of deciding strongly and weakly zero \DG for types (A), (B) and (C), see  Table \ref{tab:complex}, and their degenerated variants, see  Table \ref{tab:degcomplex},  except for strongly zero \DG for type (C) ILP, which remains open. 

We use duality gap to improve theorems, which determine the lower and the upper bound for the optimal value set of ILP, given by Rohn~\cite{Roh2006:3}. 
We weaken assumptions of finality in both theorems and extend them to equivalent statements (Thm.~\ref{thm:rohnimprove}, \ref{thm:rohn2_improve}).

\begin{table}[t]
\begin{center}
\begin{tabular}{c@{\hspace{1cm}}c@{\hspace{1.25cm}}c} 
\toprule
Type & Weakly-\ZDG & Strongly-\ZDG\\
\midrule
(A) & \npc, Thm.~\ref{thm:np:weak} &\coh, Thm.~\ref{thm:conp:strong}\\

(B)& \npc, Thm.~\ref{thm:np:weak} & \coh, Thm.~\ref{thm:conp:strong}\\

(C)&\polynomial, Thm.~\ref{thm:weak:C} & unknown\brule\\
\bottomrule
\end{tabular}
\end{center}
\caption{Computational complexity of both strongly and weakly \ZDG.}\label{tab:complex}
\end{table}
\begin{table}[t]
\begin{center}
\begin{tabular}{c@{\hspace{1cm}}c@{\hspace{1.25cm}}c} 
\toprule
Type & Weakly-\ZDG & Strongly-\ZDG\\
\midrule
degenerated (A) & \polynomial, Thm.~\ref{thm:poly:weak} &\coh, Thm.~\ref{thm:conp:strong}\\

degenerated (B)& \polynomial, Thm.~\ref{thm:poly:weak} &\coh , Thm.~\ref{thm:conp:strong}\\

degenerated (C)& \polynomial, Thm.~\ref{thm:weak:C} & \polynomial, Thm.~\ref{thm:poly:degstrong}\brule\\
\bottomrule
\end{tabular}
\end{center}
	\caption{Computational complexity of both strongly and weakly \ZDG with degenerated matrix.}\label{tab:degcomplex}
\end{table}

\section{{Duality Gap Characterizations}}

In linear programming the duality gap is not zero if and only if the primal and its dual program are infeasible. We observe that  similar conditions hold in interval linear programming. We obtain a full characterization for weakly \ZDG, and for strongly \ZDG we have only sufficient or necessary conditions. However, for the degenerate matrix there is a full characterisation of strongly \ZDG as well.

We show that there is a difference between weakly and strongly \ZDG. 
The problem in Example \ref{ex:all} below has weakly \ZDG but not strongly \ZDG. If we change $\b$ to $[-1 , -0.5]$ and $\cc$ to $[0.5 , 1]$ we obtain an~ILP which does not even have weakly \ZDG. For $\b\df[-1 , 0]$ and $\cc\df[-1 , 0]$ we obtain an~ILP which has strongly \ZDG.
Moreover, Example \ref{ex:all} provide an ILP with all possible combinations of feasibility and infeasibility of both the primal ILP and the dual ILP.

\begin{example}\label{ex:all}
Let us have the following primal and its dual ILP where $\b\df[-1 , 0]$ and $\cc\df[-1 , 1]$.
\begin{align*}
&\min~ x_1- \cc \; x_2 \mid x_1 \le \b,\,-x_2 \le -1,\, x_1, x_2 \ge 0 \\
&\max~ \b\;y_1 -y_2 \mid y_1 \le 1,\, -y_2 \le -\cc,\, y_1, y_2 \le 0.
\end{align*}
Depending on the selection of $b\in[-1 , 0]$ and $c\in[-1 , 1]$ the ILP contains scenaria where 
\begin{enumerate}[(i)]
\item both, the primal and the dual LP are infeasible: for $b\in[-1 , 0), c\in(0 , 1], $
\item the primal unbounded and the dual LP infeasible: for $b=0,\  c \in (0, 1],$
\item the dual unbounded and the primal LP infeasible: for $b \in  [-1, 0)$, $c \in [-1, 0]$,
\item both, the primal and the dual LP are feasible: for $b=0,\  c \in [-1, 0].$
\end{enumerate}
The \DG is not zero in the first case and zero in remaining cases. Thus, DG is weakly zero but it is not strongly zero. Optimal value set, however, is the same for both the primal and the dual ILP, 
$$
f(\A,\b,\cc)
=g(\A,\b,\cc)
=\{-\infty\}\cup[0,1]\cup\{\infty\}.
$$
.
\end{example}

\subsection{ Weakly zero duality gap}

We can easily observe the following full characterization of weakly zero \DG.
\begin{proposition}\label{prop:cond:weak:full}
\DG is weakly zero if and only if the primal or the dual ILP is weakly feasible. 
\end{proposition}

\begin{proof}
\DG is weakly zero if and only if there can not exist a scenario with an infeasible dual as well as primal program. That happens if and only if at least one of the primal or the dual ILP is weakly feasible.
\end{proof}

The following statement is a nice direct consequence of the Proposition \ref{prop:cond:weak:full} and two facts. First, the dual of \typeC is also of \typeC. Second, an interval system of \typeC  is weakly feasible if and only if the system $\ul A x\le\ol b,\, x\ge 0$ is feasible~\cite{Roh2006:2,Vaj1961}.

\begin{corollary}
\label{cor:cond:WZDG_C}
\TypeC $\min{\T c x \mid \A x\le\b,\, x\ge 0}$ has weakly \ZDG if and only if at least one of the linear systems $$\ul A x\le\ol b,\, x\ge 0\, \text{ or }\, \ult Ay\le\ol c,\,y\le0 $$ is feasible. 
\end{corollary}

Using weak feasibility characterisations of type (A)~\cite{Roh2006:2} and type~(B)~\cite{Ger1981,Roh2006:2} systems we obtain similar characterisations  also for type (A) and (B) ILP. However, since they are mutually dual and the weak feasibility for \typeB is expressed by an exponential reduction, we obtain more complicated characterisations.
Unless $\pP=\NP$ there is no hope for characterization by a constant number of linear systems or any polynomial-time solvable characterisation due to \NP-hardness of testing weakly \ZDG  (see section \ref{sec:comp}, Theorem \ref{thm:np:weak}).

\begin{corollary}
\label{cor:cond:WZDG_A}
\TypeA $\min{\T c x \mid \Int Ax=\Int b,\; x\ge 0}$ has weakly \ZDG if and only if at least one of the following conditions holds:
\begin{itemize}
\item[(i)] the linear system $\Dm Ax\le \Hm b,\;-\Hm Ax\le -\Dm b,\; x\ge 0$ is feasible,
\item[(ii)]  there exists $p \in \{\pm 1\}^n$ such that the linear system  $\T{( A_c - \D A\diag p)}\,y\le\Hm c,$ is feasible.

\end{itemize}
\end{corollary}

\begin{corollary}
\label{cor:cond:WZDG_B}
\TypeB $\min{\T c x \mid \Int Ax\le\Int b}$ has weakly \ZDG if and only if 
at least one of the following conditions holds:
\begin{itemize}
\item[(i)] there exist $p \in \{\pm 1\}^n$ such that the linear system  $( A_c - \D A\diag p)\,x\le\Hm b,$ is feasible.
\item[(ii)] the linear system $\Hmt Ay\le \Hm c,\;\Dmt Ay\ge \Dm c,\; y\le 0$ is feasible. 
\end{itemize}
\end{corollary}

Based on the characterizations, we obtain the following reductions for the particular types of ILP problems.

\begin{corollary}
A type (A), (B) and (C) ILP  $\min { \T \cc x\mid x \in \mathcal M(\A, \b)}$  has weakly zero DG if and only if it has  weakly zero DG, respectively, with 
\begin{enumerate}[(i)]
\item
$c:=\Hm c$ for \typeA,
\item
$b:=\Hm b$ for \typeB,
\item
$A:=\Dm A$, $b:=\Hm b$, $c:=\Hm c$ for \typeC.
\end{enumerate}
\end{corollary}

\subsection{Strongly zero duality gap}

We can easily observe the following propositions for strongly zero \DG. 
\begin{proposition}
\label{prop:cond:SZDG_str-feas}
 \DG is strongly zero if the primal program or its dual counterpart is strongly feasible.       
   \end{proposition}

\begin{proof}
Since the primal or the dual ILP is strongly feasible there exist a feasible solution for all scenarios. Thus there can not be a scenario with an infeasible primal and an infeasible dual program, i.e., \DG is strongly zero.
\end{proof}

The proposition above does not hold as an equivalence in general because an \ILP can contain just only instances with primal unbounded system and dual infeasible system, and vice versa; see Example~\ref{ex:notekv}. That ILP has strongly zero \DG, but neither primal nor dual program are strongly feasible.

\begin{example}\label{ex:notekv} Let us have a following primal and its dual program.
\begin{align*}
 &\min -y_2 \quad\mid [-1,1]\, y_1-y_2 = 1,\;y_1, y_2 \ge 0\\
 &\max~ x\quad \quad\mid [-1,1]\, x\le0,\;-x\le-1
\end{align*}

It depends on the choice of a value from the interval $[-1, 1]$. For $[-1 , 0]$ the primal program is infeasible and the dual unbounded, or vice versa for $(0, 1]$.

\end{example}

\begin{proposition} \label{prop:cond:connectivity}

If the optimal value sets $\IF$ and $\IG$ are connected and it holds that 
\begin{itemize}
\item[(i)] the primal or the dual ILP is weakly feasible, and
\item[(ii)] at least one, the primal or the dual ILP, does not contain both: an infeasible and an unbounded scenario.
\end{itemize}
then \DG is strongly zero.

\end{proposition}
\begin{proof}
Note that the assumption (i) is equal to $$\IF\neq\{\infty\} \text{ or } \IG\neq\{{-\infty}\},$$ and the assumption (ii) is equal to  $$\IF\not \supseteq \R \cup\{\infty\} \text{ or } \IG\not \supseteq \R\cup\{{-\infty}\}.$$

We prove the proposition by contradiction. Suppose that \DG is not strongly zero, thus, there exists a scenario with an infeasible both the primal and the dual program. Therefore $\infty\in f(\A,\b,\cc)$ and $-\infty\in g(\A,\b,\cc)$. Since the optimal value sets satisfy (i) and they are connected, there exists a scenario $(A, b, c)$ and a~real optimal value $r$ such that 
 $r = f(A,b,c)$ or $r=g(A,b,c)$. We obtain $r = f(A,b,c)=g(A,b,c)$ from strong duality in linear programming.
 Since the optimal value sets are connected, we have $[r,\infty]\subseteq f(\A,\b,\cc)$ and $[-\infty, r]\subseteq g(\A,\b,\cc)$. Now, we can reuse strong duality for the remaining real optimal values in $\IF$ and $\IG$. This together with connectivity leads to a contradiction with~(ii).
\end{proof}

The assumption (ii) in  Proposition \ref{prop:cond:connectivity} cannot be omitted, see Example \ref{ex:extended_real_axis}. Optimal value sets of the primal and the dual ILP are equal to the extended real axis in this example and still there is a scenario  with infeasible dual and also infeasible primal program.
\begin{example}\label{ex:extended_real_axis}
Let $\min \T \cc x \mid \A x \le\b\;$ be the following ILP
\begin{align*}
&\text{Primal ILP: } & \min [-1 , 1]x& \mid [0 , 1]x\le [-1 , 1] & &\\
&\text{Dual ILP: } & \max [-1 , 1]y&\mid [0 , 1]y=[-1 , 1],\; y\le0\\
&\text{Optimal value sets: } & \IF&=\IG=\R\cup{\{\pm\infty\}}
\end{align*}
The ILP  contains some scenaria with zero duality gap, e.g.,\ $(1 , 1 , -1),$  
and some scenaria with non-zero duality gap, e.g.,\ $(0 , -1 , -1).$  
\DG is not strongly zero even though optimal value sets are equal. 
\end{example}

We have also a necessary condition for strongly \ZDG. It gives us conditions for type (A), (B) and (C) ILP problems checkable in polynomial time. Thus, we can certify effectively that there is no strongly zero \DG in that situations.

\begin{proposition}\label{thm:gap:condit:strong3}
If an ILP has strongly zero \DG, then 
\begin{itemize}
\item[(i)] the primal ILP is weakly feasible or the dual ILP is strongly feasible, and
\item[(ii)] the primal ILP is strongly feasible or the dual ILP is weakly feasible.
\end{itemize}
\end{proposition}    

\begin{proof}
We prove it by contradiction. Since one of the primal or the dual ILP is infeasible for each scenario and the other one is not strongly feasible, there exist a scenario with both, the primal and the dual LP infeasible. Thus \DG of this scenario is not zero which is a contradiction with strongly zero \DG.
\end{proof}

The assumptions (i) and (ii) from Proposition \ref{thm:gap:condit:strong3} are not sufficient to ensure that ILP has strongly zero \DG (see Example \ref{ex:all} where an ILP and its dual are weakly feasible). Proposition \ref{thm:gap:condit:strong3} gives us together with characterisations of weak and strong feasibility of an interval system \cite{Hla2012a,Roh2006:2,Mach1970,RohKre1994} polynomially verifiable necessary conditions for strongly zero \DG.

\begin{corollary}
If a \typeA ILP $\min{\T c x \mid \Int Ax=\Int b,\; x\ge 0}$ has strongly \ZDG, then at least one of the linear systems
$$\Dm Ax\le \Hm b,\;-\Hm Ax\le -\Dm b,\; x\ge 0\; \text{ or } \;
\Hmt Ay_1-\Dmt Ay_2\le\ul c,\;y_1, y_2\ge0 $$
is feasible. 
\end{corollary}

\begin{corollary}
If a \typeB ILP $\min{\T c x \mid \Int Ax\le\Int b}$ has strongly \ZDG, then at least one of the linear systems
$$ \Hm Ax_1-\Dm Ax_2\le\ul b,\;x_1, x_2\ge0 \; \text{ or } \;
 \Hmt Ay\le \Hm c,\;\Dmt Ay\ge \Dm c,\; y\le 0$$ 
is feasible.
\end{corollary}

\begin{corollary}
If a \typeC ILP $\min{\T c x \mid \Int Ax\le\Int b,\,x\ge0}$ has strongly \ZDG, then
at least one of the linear systems
$$\ul Ax\le\ol b,\,x\ge0  \; \text{ or }\; \ult Ay \le \ul c,\,y\le0 $$
is feasible, and also one of the linear systems
$$\ol Ax\le\ul b,\,x\ge0  \; \text{ or }\; \olt Ay \le \ol c,\,y\le0 $$
is feasible.
\end{corollary}

The full characterization of strongly zero \DG remains open.

\subsection{\ILP with degenerated matrix}

Let us consider a special case of ILP --- ILP problems with degenerated matrix.
These are studied and used in practical computations since the matrix of constraints often describes a fixed structure, as in transportation problems and network flows where only capacities, right hand side, and prices, the optimization function, may vary.

We show the converse implication in Proposition \ref{prop:cond:SZDG_str-feas}  holds for  \ILP problems with degenerated matrix. This gives us a full characterization of strongly \ZDG for ILP problems with degenerated matrix. 

\begin{theorem}\label{thm:cond:deg:SZDG_strong}
    An \ILP with degenerated matrix has strongly \ZDG if and only if primal or dual programs are strongly feasible.
\end{theorem}
\begin{proof}
The backward implication holds according to Proposition \ref{prop:cond:SZDG_str-feas}.

We prove the forward implication by contradiction. Since the primal and the dual ILP are not strongly feasible, there exists a selection $b\in \Int b$ such that those scenario of the primal \ILP is infeasible. Thus each scenario $(A,b,\Int c)$ of the primal \ILP is infeasible. Analogously there exists $c \in \Int c$ such that for each scenario $(A,\Int b, c)$ the dual program is infeasible. These two selections are independent of each other. Let $(A,b,c)$ be a scenario of the \ILP, then this scenario has both dual and primal programs infeasible, i.e., duality gap of the \ILP can not be strongly zero.
\end{proof}

Theorem \ref{thm:cond:deg:SZDG_strong} together with the fact that a \typeC ILP is weakly feasible if and only if the system $\ul A x\le\ol b,\, x\ge 0$ is feasible~\cite{Fie2006,Vaj1961} give us a nice consequence for \typeC using only two classic linear programs.

\begin{corollary}\label{cor:cond:deg:SZDG_C}
A \typeC ILP $\min \Intt cx \mid Ax\le\Int b, {x\ge0}$ with degenerated matrix has strongly zero DG if and only if at least one of the systems $Ax\le\Dm b, {x\ge 0}$ and $\T Ay\le\Dm c, y\le0$ is feasible.
\end{corollary}

For completeness, we also provide characterisations for types (A) and (B).
Those are directly derived using characterizations of weak and strong feasibility of an interval linear system~\cite{Hla2012a, RohKre1994, Roh1981}.
Again unless $\pP=\NP$ there is no hope for asymptotically much better characterisations due to \coNP-hardness results for types (A) and (B) with degenerated matrix (see section \ref{sec:comp}, Theorem~\ref{thm:conp:strong}).

\begin{corollary}\label{cor:cond:deg:SZDG_A}
A \typeA ILP  $\min{ \Intt cx \mid Ax=\Int b,\; x\ge 0}$ with degenerated matrix  has strongly \ZDG if and only if 
at least one of the following conditions holds:
\begin{itemize}
\item[(i)] for each $p \in \{\pm 1\}^m$ 
 the linear system $ Ax=b_c + \diag p\, \D b,\; x\ge0$ is feasible,
\item[(ii)] the linear system $\T Ay\le\ul c$ is feasible.
\end{itemize} 
\end{corollary}

\begin{corollary}\label{cor:cond:deg:SZDG_B}
A \typeB ILP $\min{ \Intt cx\mid Ax\le\Int b}$ with degenerated matrix  has strongly \ZDG if and only if 
at least one of the following conditions holds
\begin{itemize}
\item[(i)]the linear system $ Ax \le \Dm b$ is feasible,
\item[(ii)] for each $p \in \{\pm 1\}^m$ 
 the linear system $ \T Ay=c_c + \diag p\, \D c,\; y\le0$ is feasible.
\end{itemize}
\end{corollary}

Based on the above characterizations or by direct inspection, we obtain the following reductions for the particular types of ILP problems.

\begin{corollary}
A \typeA, (B) and (C) ILP $\min { \T \cc x\mid x \in \mathcal M(A, \b)} $ with degenerated matrix  has strongly \ZDG if and only if it has  strongly \ZDG, respectively, with 
\begin{enumerate}[(i)]
\item
$c:=\Dm c$ for \typeA,
\item
$b:=\Dm b$ for \typeB,
\item
$b:=\Dm b$, $c:=\Dm c$ for \typeC.
\end{enumerate}
\end{corollary}

\section{{Computational Complexity}} \label{sec:comp}

We study computational complexity of strongly and weakly \ZDG. Here, we provide the proofs of the theorems summarized in the introduction, see Table~\ref{tab:complex} for general \ILP problems and Table~\ref{tab:degcomplex} for degenerated cases.

\subsection{Weakly \ZDG}

We show the hardness of the decision whether \DG is weakly zero.

\begin{theorem}\label{thm:np:weak}
Let $\min{\ct x} \mid \A x\le \b$ be a \typeB ILP. It is \NP-complete to decide
 whether it has weakly \ZDG .
\end{theorem} 

\begin{proof}
We construct a polynomial-time reduction from the problem of testing strong infeasibility for an interval linear program of type (B). This problem is known to be \coNP-hard~\cite{Roh2006:2}.

The \ILP $\A x\le \b$ is strongly infeasible if and only if $\A x\le \b, u\le 1$  is strongly infeasible. Let
\begin{align}
&\min{\ct x+u}\mid\A x\le\b,\,  u\le 1, \label{proof:conp:weak:primalILP}\\
&\max{\bt y+z}\mid\At y=\cc,\,  z=1,\, y,z\le 0\label{proof:conp:weak:dualILP}
\end{align} 
be a primal and its dual program. Observe that the dual program~(\ref{proof:conp:weak:dualILP})
is infeasible for every scenario. It follows that the primal program~(\ref{proof:conp:weak:primalILP}), and also the original program,
 is not strongly infeasible (and therefore it is weakly feasible) if and only if it admits weakly \ZDG.

To prove completeness, we use Corollary~\ref{cor:cond:WZDG_B} which gives us a polynomial-size certificate checkable in polynomial time. The certificate is a specification of the case (either (i) or (ii)) and additionally a $\{\pm 1\}^n$ vector $p$ for the case (i). This specifies a linear program which can be solved, and thus the certificate checked, in linear time. 
\end{proof}

Again, the same result holds for \typeA ILP as well since type (A) and (B) are mutually dual. Thus, we can use the same reduction to prove \NP-hardness. The completeness follows from Corollary~\ref{cor:cond:WZDG_A}, provided a similar certificate as in the last proof.

\begin{theorem}\label{thm:weak:C}
Let $\min{\ct x} \mid \A x\ge\b, x\ge0$ be a \typeC ILP. It is polynomial-time solvable to decide whether it has weakly \ZDG .
\end{theorem} 

\begin{proof}
The proof follows from Corollary \ref{cor:cond:WZDG_C}. It suffices to solve two classical linear programs.
\end{proof}

Moreover, testing weakly \ZDG of any type with degenerated matrix is polynomial-time solvable.

\begin{theorem}\label{thm:poly:weak}
Let be a \typeA, \typeB or \typeC ILP with degenerated matrix. It is polynomial-time solvable to decide whether it has weakly \ZDG . 
\end{theorem} 

\begin{proof}
Since testing weak feasibility of types (A), (B) and (C) ILP with degenerated matrix is polynomial~\cite{Elif} the testing weakly zero \DG is according to Proposition~\ref{prop:cond:weak:full} polynomial, too. 
\end{proof}

\subsection{Strongly \ZDG}

We follow a~similar schema of the reduction as in the weak case.

\begin{theorem}\label{thm:conp:strong}
Let $\min{\ct x} \mid \A x=\b, x\ge0$ be a \typeA ILP. It is  \coNP-hard to decide whether it has strongly \ZDG . Moreover, the problem remains \coNP-hard even for degenerated matrix $\A$. 
\end{theorem} 

\begin{proof}
We construct a polynomial-time reduction from the problem of testing strong feasibility for \typeA ILP. This problem is known to be \coNP-hard \cite{Roh2006:2} even with degenerated matrix~\cite{Elif}.

The interval system $\A x=\b,\; x\ge0$ is strongly feasible if and only if the interval system $\A x=\b,\; u-v=0,\; x,u,v\ge0$ is strongly feasible.
Let
\begin{align}
&\min{\ct x+u-2v}\mid\A x=\b,\; u-v=0,\; x,u,v\ge0, \label{proof:conp:strong:primalILP}\\
&\max{\bt y}\mid\At y\le\cc,\; z\le1,\; -z\le-2,\label{proof:conp:strong:dualILP}
\end{align} 
be a primal and its dual program. Observe that the ILP (\ref{proof:conp:strong:dualILP})
 is infeasible for every scenario. It follows that the ILP (\ref{proof:conp:strong:primalILP})
 is strongly feasible if and only if it admits strongly \ZDG.
\end{proof}

The same result holds for \typeB ILP as well since type (A) and (B) are mutually dual and then we can use the same reduction to prove \coNP-hardness. 
To proceed with the reduction it is necessary to observe that the minimization and the maximization ILP has the same complexity with respect to the determination of \ZDG. 

We remark that we cannot easily extend the hardness result to a completeness result as in the weakly \ZDG. A straightforward certificate, that is a scenario which proves a non-existence of strongly \ZDG,  does not have to be a polynomially-large certificate because it could contain an arbitrary real value.
Instead, it could be interesting to develop a characterization similar to Corollaries~\ref{cor:cond:WZDG_A} and \ref{cor:cond:WZDG_B}.

According to Corollary~\ref{cor:cond:deg:SZDG_C} it is sufficient to solve two linear programs to decide whether \DG is strongly zero for a \typeC ILP with degenerated matrix. Thus, it can be decided in polynomial time.

\begin{theorem}\label{thm:poly:degstrong}
Let be a \typeC ILP with degenerated matrix. It is polynomial-time solvable to decide whether it has strongly \ZDG.
\end{theorem}

\section{{Improvements of Duality Theorems Using Duality Gap}}

We can use \DG to generalize one of the basic theorems about lower and upper bounds for the optimal value set of ILP, which is a form of strong duality. 

\begin{theorem}[\cite{Roh2006:3}] \label{thm:Rohn} 
Let $\min \ct x\mid \A x=\b,\; x\ge0 $ be a~type~(A) ILP. 
If $\IHmF$ is finite then (\ref{thm:horni_mez_Rohn}) holds, 
if $\IDmF$ is finite then  (\ref{thm:dolni_mez_Rohn}) holds.
\begin{align}
&\IHmF=\maxx{\left\{\maxx_{b \in \b }{\ \T by} \mid \exists A \in \A, \exists c \in \cc: \T Ay \le c \right\}}.\label{thm:horni_mez_Rohn}\\
&\IDmF=\maxx{\left\{\minn_{b \in \b }{\ \T by} \mid \forall A \in \A, \forall c \in \cc: \T Ay \le c \right\}}. \label{thm:dolni_mez_Rohn}
\end{align}
\end{theorem}

Notice that equations (\ref{thm:horni_mez_Rohn}) and (\ref{thm:dolni_mez_Rohn})  can be equivalently written as 
\begin{align*}
&\IHmF=\maxx{\left\{\maxx_{b \in \Int b }{\left\{{\T by}\right\} }\mid y\in \bigcup_{A \in \Int A,\, c \in \Int c}{\left\{y\colon \T Ay \le c\right\}} \right\}},\\
&\IDmF=\maxx{\left\{\minn_{b \in \Int b}{\left\{ \T by \right\}} \mid y\in \bigcap_{A \in \Int A,\, c \in \Int c}{\left\{y\colon\T Ay \le c\right\}} \right\}}.
\end{align*}

We show that the assumption of finiteness in Theorem \ref{thm:Rohn} can be weakened. We also obtain a~full characterization of $\IHmF$ and $\IDmF$ and show that strongly \ZDG is a sufficient assumption to have the equalities (\ref{thm:horni_mez_Rohn})~and~(\ref{thm:dolni_mez_Rohn}).

\begin{theorem} \label{thm:rohnimprove}
Let $ \min{\ct x} \mid \A x=\b, x\ge0$ be a \typeA ILP.  Then 
\begin{align} \label{thm:int:horni_mez_Rohn_improve}
\IHmF=\maxx{\left\{(\maxx_{b \in \Int b }{\ \T by}) \mid \exists A \in \Int A, \exists c \in \Int c: \T Ay \le c \right\}} 
\end{align}
if and only if the ILP satisfies at least one of the conditions: 
\begin{itemize}
\item[(i)] it has strongly zero \DG ,
\item[(ii)] $\IHmG=\infty$, i.e., the dual ILP has an unbounded scenario.
\end{itemize} 
\end{theorem}

\begin{proof}
Let us denote $\ol f\df\IHmF$ and $\ol g\df\IHmG$. We show ($\ref{thm:int:horni_mez_Rohn_improve}$) says that $\ol f=\ol g$. We have 
\begin{align*}
\ol g
&=\maxx_{A \in \A,\; b \in \b,\; c \in \cc }{\left\{(\maxx{\ \T by}) \mid  \T Ay \le c \right\}} \\
&=\maxx\left\{(\maxx_{b \in \Int b }{\ \T by}) \mid \exists A \in \Int A, \exists c \in \Int c: \T Ay \le c \right\}.
\end{align*}

We prove the backward implication. For ILPs which satisfy the assumption~(i) we use strong duality (Proposition \ref{claim:strong_duality}) which states $\ol f$=$\ol g\, $. For assumption (ii), we use weak duality (Proposition~\ref{claim:weak_duality}) which states $\ol f\ge\ol g$, to show $\ol f=\infty=\ol g$.

We prove the forward implication by contradiction. Since \DG is not strongly zero, there exist a scenario with infeasible both primal and dual program. Thus $\ol f=\infty$.  Since $\ol f=\ol g$, we obtain $\ol g=\infty$. This is a contradiction (with the negation of (ii)).
\end{proof}

We are not using the properties of \typeA in the proof of Theorem \ref{thm:rohnimprove}, except of the property that there is the dual form of the given ILP in the right hand side of (\ref{thm:int:horni_mez_Rohn_improve}). Thus Theorem \ref{thm:rohnimprove} with the appropriate modification:
$$ 
\IHmF=\maxx{\left\{(\maxx_{b \in \Int b }{\ \T by}) \mid \exists A \in \Int A, \exists c \in \Int c: y\in {\mathcal{N}(A,c)} \right\}} 
$$ 
where $\mathcal{N}(A,c)$ is the feasible set of the dual program, 
 holds for other types as well.

Weakly \ZDG is not sufficient to fulfill equation (\ref{thm:int:horni_mez_Rohn_improve}), see Example~\ref{ex:all} with changed intervals $\b$~and $\cc$ to $\b\df[-1 , 0 ]$ and $\cc \df[0.5 , 1]$. The upper bounds of optimal value sets are $\IHmF=\infty$ and $\IHmG=-\infty$.

\begin{theorem}\label{thm:rohn2_improve}
Let $ \min{\ct x} \mid \A x=\b, x\ge0$ be a \typeA ILP.  Then 
\begin{align} \label{thm:int:dolni_mez_Rohn_improve}
\IDmF=\maxx{\left\{(\minn_{b \in \b }{\ \T by}) \mid \forall A \in \A, \forall c \in \cc: \T Ay \le c \right\}}
\end{align}
if and only if the primal ILP is weakly feasible or the dual ILP is strongly feasible.
    \end{theorem}

    \begin{proof}
    According to \cite{Hla2012a,Roh2006:3}, we have
    \begin{align}\label{pf:dolni_mez_Rohn_improve:uf}
    \IDmF=\min{\ul{c}^Tx \mid \ul{A}x\leq\ol{b},\ -\ol{A}x\leq-\ul{b},\ x\geq0 }.
    \end{align}
    Its dual program reads
    \begin{align}\label{pf:dolni_mez_Rohn_improve:g}
    g=\max{\ol{b}^Ty_1-\ul{b}^Ty_2 \mid 
      \ul{A}^Ty_1-\ol{A}^Ty_2\leq\ul{c},\ y_1,y_2\leq0 },
      \end{align}
      or, equivalently
      \begin{align*}
      g=\max{\alpha  \mid \ol{b}^Ty_1-\ul{b}^Ty_2\geq\alpha,\ 
        \ul{A}^Ty_1-\ol{A}^Ty_2\leq\ul{c},\ y_1,y_2\leq0 }.
        \end{align*}
        By substitution $y_1:=-y_1$ and $y_2:=-y_2$, we obtain
        \begin{align*}
        g=\max{\alpha  \mid \ol{b}^Ty_1-\ul{b}^Ty_2\leq-\alpha,\ 
          \ol{A}^Ty_2-\ul{A}^Ty_1\leq\ul{c},\ y_1,y_2\geq0 }.
          \end{align*}
          The constraints describe strong solutions of the corresponding interval linear inequalities, so by \cite{Roh2006:2}, we can equivalently write
          \begin{align*}
          g=\max{\alpha  \mid -b^Ty\leq-\alpha\ \forall b\in\b,\ 
            A^Ty\leq c\ \forall A\in\A,\forall c\in\cc},
            \end{align*}
            which yields the right-hand side of (\ref{thm:int:dolni_mez_Rohn_improve}).
            Equation (\ref{thm:int:dolni_mez_Rohn_improve}) is valid if and only if there is zero duality gap between (\ref{pf:dolni_mez_Rohn_improve:uf}) and (\ref{pf:dolni_mez_Rohn_improve:g}). That is, either the primal program~(\ref{pf:dolni_mez_Rohn_improve:uf}) must be feasible (i.e., the ILP is weakly feasible), or the dual program~(\ref{pf:dolni_mez_Rohn_improve:g}) must be feasible (i.e., the dual ILP is strongly feasible).
\end{proof}

Proposition \ref{thm:gap:condit:strong3} can be used to obtain the following corollary. 
\begin{corollary}
If a \typeA ILP has strongly zero \DG, then (\ref{thm:int:dolni_mez_Rohn_improve}) holds.
\end{corollary}

 Weakly zero \DG is not sufficient to fulfill the equation~(\ref{thm:int:dolni_mez_Rohn_improve}), see the Example~\ref{ex:A_wzDG_primal_infeasible}.
 \begin{example}\label{ex:A_wzDG_primal_infeasible}
Let $\min \T \cc x \mid \A x =\b, x\ge0$ be the following \typeA ILP 
\begin{align*}
&\text{primal ILP:} & &\min [-1, 0]\;x_1 & &\mid x_1-x_2=0,\,x_1-x_2=1,\,x_1,x_2\ge0\\
&\text{dual ILP:} & &\max y_2 & &\mid y_1+y_2\le[-1, 0],\, -y_1-y_2\le0
\end{align*}
The primal ILP is infeasible for each scenario. Thus $\IDmF=\infty$ which is the left hand side of (\ref{thm:int:dolni_mez_Rohn_improve}). The dual ILP is weakly feasible but not strongly feasible. Thus, the \DG is weakly zero. Since the dual ILP has no strong solution, the right hand side of (\ref{thm:int:dolni_mez_Rohn_improve}) is equal to$\max{\emptyset}=-\infty$.
\end{example}

\section{Conclusions}
We introduced weakly and strongly zero duality gap in interval linear programming. Mostly, we focused on standard types of interval linear systems, \typeA with equalities and non-negative variables, \typeB with inequalities and \typeC with equalities and non-negative variables, and their variants with a real matrix of constraints.

 We proposed a characterisation for weakly zero duality gap which uses two linear programs for \typeC and exponentially many linear programs for types (A) and (B). We provided sufficient conditions and necessary conditions for strongly zero duality gap and a full characterisation for ILP problem with real matrix of coefficients. We analyzed computational complexity of determining whether an ILP problem has weakly or strongly zero duality gap.

We focused on the best and the worst case for the optimal value of an~ILP problem. We strengthened previous results and we provide equivalent characterisations.

\medskip

We state open questions that have arisen:
\begin{enumerate}[-]
\item Determine the computational complexity of testing strongly zero \DG of \typeC ILP.
\item Provide a full characterization of strongly zero \DG.
\end{enumerate}

 Another interesting problem is to find a probably super-polynomial algorithm that can decide whether \DG is strongly zero for those problems that were proven to be \coNP-hard. 
This could also yield completeness results.
In~further direction an approximation algorithm or a heuristic could be of some interest.

\bibliographystyle{plain}
\bibliography{src/literatura}

\end{document}